\documentclass[11pt]{article}

\usepackage{amsmath, amsthm, amsfonts}

\usepackage[utf8]{inputenc}
\usepackage[english]{babel}
\usepackage{amssymb}
\usepackage{graphicx}
\usepackage{epstopdf} % Windows
\usepackage{fancyhdr}
\usepackage{amsthm}

\usepackage{soul}
\setstcolor{red}

\usepackage{subfig}

\usepackage{authblk}

\usepackage{xcolor}

\usepackage{natbib}

\usepackage[margin=1 in]{geometry}

\newtheorem{proposition}{Proposition}
\newtheorem{corollary}{Corollary}
\newtheorem{lemma}{Lemma}
\newtheorem{example}{Example}
\theoremstyle{remark}
\newtheorem*{remark}{Remark}
\usepackage{mathptmx}

\usepackage{setspace}
\title{Stochastic Comparisons of Two-Units Repairable Systems}
\author[a]{Josu\'e~M.~Corujo\thanks{j.corujo@matcom.uh.cu (Corresponding author)}}
\author[a]{Jos\'e~E.~Vald\'es\thanks{vcastro@matcom.uh.cu }}
\author[b]{Juan~C.~Laria\thanks{jlaria@est-econ.uc3m.es}}
\affil[a]{Facultad de Matem\'atica y Computaci\'on, Universidad de La Habana, Cuba}
\affil[b]{Departamento de Estad\'istica, Universidad Carlos III de Madrid, Espa\~na}
\date{}

\begin{document}
\maketitle

{\bf Corresponding author:}

Josu\'e M. Corujo, Facultad de Matem\'atica y Computaci\'on, Universidad de La Habana, Cuba

{\bf Short tittle:}

Stochastic Comparisons of Repairable Systems

\abstract{We consider two models of two-units repairable systems: cold standby system and warm standby system. We suppose that the lifetimes and repair times of the units are all independent {exponentially distributed} random variables. Using stochastic orders we compare the lifetimes of systems under different assumptions on the parameters of exponential distributions.
{We also consider a cold standby system where the lifetimes and repair times of its units are not necessarily exponentially distributed.}}


\section{Introduction}

The existing literature on stochastic comparison of systems is extensive and focuses particularly on systems without reparation. Many recent studies (e.g. \cite{peng_zhao_optimal_2012}, \cite{peng_zhao_allocation_2013} and \cite{wang2017}) deal with the case where the units have exponentially distributed lifetimes. In contrast, other authors are more concerned with the case of non-exponential lifetimes (e.g. \cite{jose_e._valdes_optimal_2003}, \cite{jose_e._valdes_optimal_2006}, \cite{gerandy_brito_hazard_2011}, \cite{nandaCSTM}, \cite{chenCSTM} and \cite{kundu2017}).
Whilst some research has been carried out on this topic, no one as far as we know has addressed stochastic comparisons of repairable systems.

This paper focuses on two-unit repairable systems. A detailed introduction to these models can be found in Chapter 7 of \cite{boris_v._gnedenko_probabilistic_1995}. In particular, we consider two cases: warm standby systems (when the spare unit can fail) and cold standby systems (when the spare unit cannot fail).
An exhaustive analysis of these systems was carried out by \cite{tadashi_dohi_stochastic_2002} and the references therein, which included a summary of the earliest research on these models, some of them dealing with extensions of their results to real life systems.

In general, to find an analytic expression for the probability distribution of the lifetime of a repairable system could be impossible in the non-Markovian case. Unfortunately, when the system has more than two units, the density function of its lifetime could have a very complex expression, even in the Markovian case. This limitation makes it extremely difficult to compare the lifetimes of systems with more than two units. Several studies have established interesting properties for the reliability of two-units standby systems. For instance, \cite{li93} proved the convergence of the residual lifetime of a Markovian repairable system to the exponential distribution, and \cite{bao2012study} investigated a generalization of cold standby systems. Despite this interest, no one to the best of our knowledge has studied stochastic orderings of two-unit standby repairable systems. With this in mind, the primary aim of this paper is to compare lifetimes of two-unit repairable systems using stochastic orders.

The survival function of a random variable with distribution function $F(t)$ will be denoted by $\overline{F}(t) = 1 - F(t)$. The terms increasing and decreasing will be used in the non-strict sense.
Also, the Laplace transform of a real function $h(t)$ will be denoted by $\widehat{h}(s)$.
Given a nonnegative random variable $X_i$, and $\overline{F}_i(t)$, $f_i(t)$ and $r_i(t)$ its survival, probability density and hazard rate functions, respectively, for $i=1,2$, we consider the following definitions of stochastic orders. $X_1$ is said to be smaller than $X_2$ in the
\begin{enumerate}
	\item {\textit{Laplace transform order} (denoted as $X_1 \leq_{lt} X_2$), if $\widehat{\overline{F}}_1 (s) \leq \widehat{\overline{F}}_2(s)$, for all nonnegative $s$,}
	\item \textit{Increasing concave order} (denoted as $X_1 \leq_{icv} X_2$), if \(\displaystyle \int_0^{t} {F}_1(x)dx \ge \displaystyle \int_0^{t} {F}_2(x)dx\),
	%\item \textit{Increasing convex order} (denoted by $X \le_{icx} Y$,) if \(\displaystyle \int_t^{\infty} \overline{F}_1(x)dx \le \displaystyle \int_t^{\infty} \overline{F}_2(x)dx.\),
	\item \textit{Usual stochastic order} (denoted as $X_1 \leq_{st} X_2$), if $\overline{F}_1 (t) \leq \overline{F}_2(t)$, for all nonnegative $t$,
	\item \textit{Hazard rate order} (denoted as $X_1 \leq_{hr} X_2$), if $r_1 (t) \geq r_2(t)$, for all nonnegative $t$,
	\item \textit{Likelihood ratio order} (denoted as $X_1 \leq_{lr} X_2$), if $f_1(t)/f_2(t)$ is decreasing for all nonnegative $t$.
\end{enumerate}

The relationship between these stochastic orders is well known ($5.\Rightarrow 4. \Rightarrow 3. \Rightarrow 2. \Rightarrow 1.$).  \cite{shaked_stochastic_2007} and \cite{belzunce2015introduction} provide in-depth analysis of stochastic orders.

Suppose that we wanted to compare the lifetimes of two two-units standby systems, say, system $1$ and system $2$. Intuitively, if the lifetimes of units of system $1$ are ``greater'' than those of system $2$, and the repair times of system $1$ are ``lesser'' than those of system $2$, then  the lifetime of system $1$ should be ``greater''  than the lifetime of system $2$. The extent to which ``greater'' and ``lesser'' can be understood, and how to proceed in non-intuitive cases, remain unclear.

In this context, Proposition \ref{WSS_rep_fix} asserts that if two {Markovian} two-units  warm standby systems have stochastically equal lifetimes for their units and the hazard rates of the repair times of system $1$ are greater than those of system $2$ then the lifetime of system $1$ is greater than the lifetime of system $2$ in the hazard rate order. However this ordering does not hold in the likelihood ratio order. Proposition \ref{CSS_fix_func} verifies an analogous result for {Markovian} two-units cold standby systems.
These propositions suggest that, even under intuitive conditions, it is unclear in what sense the lifetime of one system is greater than the lifetime of the other.
Moreover, Example \ref{nenita3} shows that a likelihood ratio ordering can be obtained when the lifetime of the principal unit of the system with stochastically greatest lifetime has greater hazard rate than the respective lifetime of the system with stochastically smallest lifetime. This example suggests that stochastic orderings can be established, even under non intuitive hypothesis.
We believe that these results could be exploited in order to make decisions to improve the reliability of systems.

It would be interesting to find necessary and sufficient conditions to establish stochastic orderings between two-units standby systems. In this sense, Propositions \ref{WSS_fun_fix_ssi} and \ref{nenita2} provide necessary and sufficient conditions for a likelihood ratio ordering to hold between the lifetimes of two warm and cold standby systems.

This paper is organized as follows. Section \ref{sec::models} describes the models corresponding with the aforementioned cold and warm standby systems. Sections \ref{subsecc::warm} and \ref{subsec::cold_markovian} highlight the key theoretical results on Markovian warm and cold standby systems. Regarding the \emph{Increasing Likelihood Rate} ($ILR$) and \emph{Decreasing Likelihood Rate} ($DLR$) classes, Subsection  \ref{subs::ageing_classes} investigates whether the lifetimes of the Markovian warm and cold standby systems considered here, belong to these aging classes.
The final section takes a further step towards developing realistic models, assuming general distributions for the lifetimes and repair times of components in the cold standby system. These result could potentially lead to attractive applications in real-life systems.

\section{The models}\label{sec::models}

This section describes the models discussed in this investigation.
As explained in the introduction, we consider two models of systems, \emph{cold standby systems} and \emph{warm standby systems}. Both systems are composed of two units, allocated in two possible slots, namely principal and standby positions.
Throughout this paper, the phrase \emph{principal (spare or standby) unit} will be used to refer to units allocated in the \emph{principal (stanby)} position. Furthermore, both models include a repair unit which cannot fail.

In warm standby systems, the standby unit is understood as a warm spare, because it could fail. On the other hand, in cold
standby systems, the standby unit cannot fail and it is therefore understood as a cold spare. The following are some basic assumptions for both systems.
\begin{itemize}
	\item When a unit fails, it is immediately sent to the repair unit.
	\item If the principal unit fails, the standby unit (if not under reparation) takes its place (i.e. the spare unit becomes the principal unit).
	\item As soon as the reparation of a unit completes, the unit assumes the standby position (unless the system has already failed).
	\item Lifetimes and repair times of units are independent random variables.
	\item Reparations of units are perfect, i.e. right after its reparation finishes, the unit is restored to a state equivalent to a new unit.
	\item The system failure occurs when the principal unit fails and the other unit is under reparation.
\end{itemize}

In summary, when the cold standby system works, both units are alternately operating, and when a unit is under repair the other unit works as principal.
On the other hand, in the warm standby system, both units operate simultaneously, and the first of them that fails is immediately sent to the repairing unit. The other unit, which is still operational, is allocated in the principal position (if it was not already there). It could be possible that one unit consecutively fails and is repaired several times, while the other unit is working.
In both cold and warm standby models, the system failure occurs in the first instant in which both units are down.

In the section that follows, we will explore properties of a Markovian warm standby system, assuming that the unit in standby position could fail after an exponentially distributed time, independent from the lifetimes and repair times of the other unit.
In addition, Section \ref{subsec::cold_markovian} will examine a Markovian cold standby system where the lifetimes and repair times of both units are equally distributed.
In Sections \ref{subsecc::warm} and \ref{subsec::cold_markovian} the lifetimes and repair times of the units will depend on their allocated positions, and not on the units themselves. On the other hand, in Section \ref{2unitGral} we will undertake an analysis of cold standby systems where the lifetimes and repair times of the two units are not necessarily exponentially or equally distributed.

%In a \emph{warm standby system} the standby unit is understood as a warm spare \textcolor[rgb]{1.00,0.00,0.00}{because it could fail}.
%\textcolor{red}{The assumptions for the \emph{cold standby system} are the same, except that the standby} unit cannot fail and \textcolor[rgb]{1.00,0.00,0.00}{therefore} it is understood as a cold spare.

%\section{Results}\label{sec::results}

\section{{Markovian} warm standby system}\label{subsecc::warm}

Let $\tau^{W}$ be the lifetime of a  {warm standby system} with exponentially distributed lifetimes with means $1/\lambda_{1}$ for the principal unit and $1/\lambda_{2}$ for the secondary unit, and with exponentially distributed repair time with mean $1/\mu$. Let $\Phi^{W}(t)$ denote the survival function of $\tau^{W}$. Then, the states of the system at time $t$, denoted by $X^C(t)$, are defined as,
\[
X^C(t) =
\left\{
	\begin{array}{ccl}
    0 & \text { if } & \text{ at time }  $t$ \text{ one unit is working and the other is in standby position,}\\
    1 & \text { if } & \text{ at time } $t$ \text{ one unit is working and the other is under repair,}\\
    2 & \text { if } & \text{ at time } $t$ \text{ the two units have failed and thus the system has failed.}
	\end{array}\right.
\]

Therefore, the system states space is $E = \{0,1,2\}$. If at a given time $t$, it holds that $X^C(t) \in \{0,1\}$, then the system is functioning at time $t$, otherwise if $X^C(t) \in \{2\}$, then the system has failed at (or before) time $t$. We consider $\{2\}$ an absorbing state.
The stochastic process $\{X(t),\, t\ge 0\}$ is an homogeneous continuous-time Markov process with $Q$-matrix
\[
	Q =
	\begin{pmatrix}
    -(\lambda_1 + \lambda_2) & \lambda_1 + \lambda_2 & 0\\
    \mu &  -(\lambda_1 + \mu) & \lambda_1\\
    0 & 0 & 0
\end{pmatrix}.
\]
Let $(P_{ij}(t))$, {for $i,j = 0,1,2$,} be the transition probability matrix {of this process}.
Then, since $\Phi^{W}(t) = P_{00}(t) + P_{01}(t)$, it can be obtained that (see, e.g. \cite{boris_v._gnedenko_probabilistic_1995} pages 263-264),

% \textcolor{red}{Let $\tau^{W}$ the lifetimes of a {warm standby systems} with expected lifetimes $1/\lambda_{1}$ for the principal unit and  $1/\lambda_{2}$ for the secondary unit, and with expected repair time $1/\mu$. Let $\tau^{W*}$ be the lifetime of a warm standby system when at the initial instant one unit is working and the other is under repair.  Let us denote by $\Phi^{W}(t)$ and $\Phi^{W*}(t)$ the survival functions of $\tau^{W}$ and $\tau^{W*}$, respectively. Then the following system of integral equations holds,
% \begin{eqnarray*}
% 	\Phi^{W}(t) &=& e^{-\lambda_1 t} + \lambda_1 \int_0^t \Phi^{W*}(t - x) e^{-(\lambda_1 + \lambda_2) x} dx,\\
% 	\Phi^{W*}(t) &=& e^{-(\lambda_2 + \mu) t} + \mu \int_0^t \Phi^{W}(t - x) e^{-(\lambda_2 + \mu) x} dx.
% \end{eqnarray*}}

% \textcolor{red}{Applying the Laplace tranformation and solving the resulting linear system it is got
% \[\widehat{\Phi}^{W}(s) = \frac{s + 2\lambda_1 + \lambda_2 + \mu}{s^2 + s(2\lambda_1 + \lambda_2 + \mu) + \lambda_1(\lambda_1 + \lambda_2)}.\]}

% \textcolor{red}{Thus, inverting this Laplace transform it is obtained the expression of $\Phi^{{W}}(t)$ (see, e.g. \cite{boris_v._gnedenko_probabilistic_1995} pages 263-264),}

\begin{eqnarray}
	\Phi^{{W}}(t) &=& \exp\left\{-\frac{(2\lambda_1+\lambda_2+\mu)t}{2}\right\}\left[\cosh\left(\frac{at}{2}\right)  + \frac{2\lambda_1+\lambda_2+\mu}{a}\sinh\left(\frac{at}{2}\right)\right],\label{distr_WSS}
\end{eqnarray}
where $a=\sqrt{(\lambda_2+\mu)^2+4\lambda_1\mu}$.

Let $\phi^{W}(t)$ and $r^{W}(t)$ be the probability density function and the hazard rate function of $\tau^W$, respectively. Then,
\begin{eqnarray}
	\phi^{{W}}(t) &=& \exp\left\{-\frac{(2\lambda_1+\lambda_2+\mu)t}{2}\right\} {\frac { 2\lambda_1\left( \lambda_1+\lambda_2 \right) }{a}}\sinh \left( \frac{at}{2} \right),\label{dens_WSS}\\
	r^{{W}}(t) &=& \frac{2\lambda_1(\lambda_1+\lambda_2)}{2\lambda_1+\lambda_2+\mu+a\coth(\frac{at}{2})}\label{hazard_WSS}.
\end{eqnarray}

We say that two {real-valued} functions $h_1(x)$ and $h_2(x)$ are equal in sign, and it is denoted by $h_1(x) =_{sg} h_2(x)$, if there is a strictly positive function $h(x)$, such that, $h_1(x) = h(x)h_2(x)$.  First, we are going to prove {a} lemma, which will be used later.

\begin{lemma}\label{LEMON}
	If $c_1 \geq c_2 \geq 0$, then the functions $\displaystyle \frac{\cosh(c_1 t)}{\cosh(c_2 t)}$, $\displaystyle \frac{\sinh(c_1 t)}{\sinh(c_2 t)}$ and $\displaystyle \frac{\coth(c_1 t)}{\coth(c_2 t)}$ are increasing for $t\geq 0$.
\end{lemma}

\begin{proof}

We give a proof of this result only for $\displaystyle \frac{\cosh(c_1 t)}{\cosh(c_2 t)}$. The proofs for $\displaystyle \frac{\sinh(c_1 t)}{\sinh(c_2 t)}$ and $\displaystyle \frac{\coth(c_1 t)}{\coth(c_2 t)}$ are similar.

Let $\displaystyle f(t)=\frac{\cosh(c_1 t)}{\cosh(c_2 t)}$, then
\[f'(t)=_{sg} c_1 \sinh(c_1t)\cosh(c_2t)-c_2 \sinh(c_2t)\cosh(c_1t).\]
Using properties of the hyperbolic functions, we obtain
\begin{eqnarray*}
	f'(t)&=_{sg}& c_1 \{\sinh[(c_1+c_2)t]+\sinh[(c_1-c_2)t]\} - c_2\{\sinh[(c_1+c_2)t]+\sinh[(c_2-c_1)t]\} \\
	&=& \sinh[(c_1+c_2)t][c_1-c_2] + \sinh[(c_1-c_2)t][c_1+c_2] \geq 0,
\end{eqnarray*}
for all $t\geq 0$.
\end{proof}

Let $\tau^{W}_i$ be the lifetime of a {Markovian} {warm standby system} with expected lifetimes $1/\lambda_{i1}$ for the principal unit and  $1/\lambda_{i2}$ for the secondary unit, and with expected repair time $1/\mu_i$, for $i=1,2$, res\-pectively. Let $\phi^W_i(t)$ denote the probability density function of $\tau^W_i$, for $i=1,2$. Our goal is to establish stochastic orders on the lifetimes of these systems using relations between the parameters of their distributions. Let us {denote} $a_i=\sqrt{(\lambda_{i2}+\mu_i)^2+4\lambda_{i1}\mu_i}$, for $i=1,2$, and {define} the function $(x)^+ = \max\{x, 0\}$.

\begin{proposition}\label{WSS_fun_fix_ssi}
$\tau^{W}_1 \geq_{lr} \tau^{W}_2$  if and only if  $2(\lambda_{21} - \lambda_{11}) - \lambda_{12} + \lambda_{22} + \mu_2 - \mu_1 \geq (a_2 - a_1)^+.$
\end{proposition}

\begin{proof}

Consider the ratio between the probability density functions of $\tau^{W}_1$ and $\tau^{W}_2$. Using (\ref{dens_WSS}) we obtain
\[\frac{\phi^{W}_1(t)}{\phi^{W}_2(t)} =K \frac{\exp\left\{ \frac{2 \lambda_{21} + \lambda_{22} + \mu_2}{2} t \right\} \sinh(\frac{a_1 t}{2}) }{ \exp\left\{ \frac{ 2 \lambda_{11} + \lambda_{12} + \mu_1 t}{2} \right\} \sinh(\frac{a_2 t}{2})},\]
where $K$ is a positive constant.
Then
\begin{eqnarray*}
	\left( \frac{\phi^{W}_{1}(t)}{\phi^{W}_{2}(t)} \right)' &=_{sg}& h(t) = (2\lambda_{21} + \lambda_{22} + \mu_2) + a_1 \coth\left( \frac{a_1}{2} t \right)  - ( 2 \lambda_{11} + \lambda_{12} + \mu_1) - a_2 \coth\left( \frac{a_2}{2} t \right).
\end{eqnarray*}

 Deriving {$h(t)$}, it can be checked that {this function} is increasing (decreasing) if and only if
 \[
 	\frac{\sinh\left(\frac{a_1}{2} t\right)}{\sinh\left(\frac{a_2}{2} t\right)} \leq (\geq) \frac{a_1}{a_2} = \lim\limits_{t \rightarrow 0}\, \frac{\sinh\left(\frac{a_1}{2} t\right)}{\sinh\left(\frac{a_2}{2} t\right)},
 \]
 and using Lemma \ref{LEMON} we obtain that $h(t)$ is monotone.  Then, a necessary and sufficient condition for $h(t)$ to be positive is that the following inequalities hold
\[
\begin{array}{ccl}
\lim\limits_{t \rightarrow \infty} h(t) &=& (2\lambda_{21} + \lambda_{22} + \mu_2) + a_1 - ( 2 \lambda_{11} + \lambda_{12} + \mu_1) - a_2 \ge 0,\\
\lim\limits_{t \rightarrow 0} h(t) &=& (2\lambda_{21} + \lambda_{22} + \mu_2) - ( 2 \lambda_{11} + \lambda_{12} + \mu_1) \ge 0.
\end{array}
\]
Therefore, $\tau^{W}_1 \geq_{lr} \tau^{W}_2$ if and only if $2(\lambda_{21} - \lambda_{11}) - \lambda_{12} + \lambda_{22} + \mu_2 - \mu_1 \geq (a_2 - a_1)^+.$
\end{proof}

Using Proposition \ref{WSS_fun_fix_ssi} we get a sufficient (not necessary) condition for $\tau^{W}_1 \geq_{lr} \tau^{W}_2$ to hold.

\begin{corollary}\label{WSS_fun_fix}
	If $\mu_1 \ge \mu_2$, $\lambda_{11} \leq \lambda_{21}$ and
	\begin{equation}\label{condt1}
		2(\lambda_{21} - \lambda_{11}) - \lambda_{12} + \lambda_{22} + \mu_2 - \mu_1 \ge 0,
	\end{equation}
	then $\tau^{W}_1 \geq_{lr} \tau^{W}_2$.
\end{corollary}

\begin{proof}

Let us define $c_i = 2\lambda_{i1} + \lambda_{i2} + \mu_i$, for $i=1,2$. Due to Proposition \ref{WSS_fun_fix_ssi}, as (\ref{condt1}) holds, it is sufficient to check the inequality $c_2 - c_1 \geq a_2 - a_1.$
Since
\begin{eqnarray}
a_2^2 - a_1^2 &=& (\lambda_{22} + \lambda_{12} + \mu_2 + \mu_1)(\lambda_{22} + \mu_2 - \lambda_{12} - \mu_1) + 4(\lambda_{21} \mu_2 - \lambda_{11}\mu_1),\label{expr1a}\\
&=& (\lambda_{22} + \lambda_{12} + \mu_2 + \mu_1)[c_2 - c_1 -2(\lambda_{21} - \lambda_{11})] + 4(\lambda_{21} \mu_2 - \lambda_{11}\mu_1).\label{expr1}
\end{eqnarray}
we obtain
\begin{eqnarray}
(c_2 - c_1 + a_1 - a_2)(a_1 + a_2) &=& (c_2 - c_1)(a_1 + a_2) + a_1^2 - a_2^2\nonumber\\
	&=& (c_2 - c_1)[a_1 + a_2 - (\lambda_{22} + \lambda_{12} + \mu_2 + \mu_1)]\nonumber\\
	& & + 2(\lambda_{21} - \lambda_{11})(\lambda_{22} + \lambda_{12} + \mu_2 + \mu_1) - 4(\lambda_{21}\mu_2 - \lambda_{11}\mu_1)\label{lasttime1}\\
	&\ge& 2(\lambda_{21} - \lambda_{11})(\lambda_{22} + \lambda_{12}) + 2\lambda_{21}(\mu_1 - \mu_2) - 2\lambda_{11}(\mu_2 - \mu_1)\label{lasttime2}\\
	&=& 2(\lambda_{21} - \lambda_{11})(\lambda_{22} + \lambda_{12}) + 2(\mu_{1} - \mu_{2})(\lambda_{21} + \lambda_{11})\label{expr2}\\
	&\ge& 0,\nonumber
\end{eqnarray}
where the equality (\ref{lasttime1}) comes from (\ref{expr1}), the inequality (\ref{lasttime2}) is true due to $a_i \ge \mu_i + \lambda_{i2}$, for $i=1,2$, and (\ref{expr2}) is nonnegative due to $\lambda_{11} \le \lambda_{21}$ and $\mu_1 \ge \mu_2$. Thus $c_2 - c_1 \ge a_2 - a_1$ holds and consequently $\tau^W_1 \ge_{lr} \tau^W_2$.
\end{proof}

We will refer to the warm standby system with lifetime $\tau^{W}_i$ as the system $i$, for $i=1,2$.
{From Corollary \ref{WSS_fun_fix} we can state some intuitive results which are not that easy to infer from Proposition \ref{WSS_fun_fix_ssi}.}
For example, it is not difficult to prove that $\tau^{W}_1 \geq_{lr} \tau^{W}_2$ holds
when the hazard rate of the lifetimes of the units of system $1$ are smaller than the hazard rate of the ones of system $2$, and the repair times are  stochastically equal for both systems, i.e. $\lambda_{11} \leq \lambda_{21}$, $\lambda_{12} \leq \lambda_{22}$ and $\mu_1 = \mu_2$.
{
{Suppose that we have two units with hazard rates $\lambda_1$ and $\lambda_2$, respectively, and we want to decide which one to put on the principal position}. Thus, we want to compare $\tau^W_1$ and $\tau^W_2$ when $\lambda_{11} = \lambda_{22} = \lambda_1$ and $\lambda_{12} = \lambda_{21} = \lambda_2$. Using Corollary \ref{WSS_fun_fix} we obtain the relation $\tau^W_1 \ge_{lr} \tau^W_2$ when $\lambda_2 - \lambda_1 \ge \mu_1 - \mu_2 \ge 0$. If both systems have stochastically equal repair times $\tau^W_1 \ge_{lr} \tau^W_2$ is equivalent to $\lambda_1 \le \lambda_2$. So, it is better the system with the smallest hazard rate in the principal position.}

The following example shows that if $\mu_1 = \mu_2$ we can obtain $\tau^{W}_1 \geq_{lr} \tau^{W}_2$ even when $\lambda_{11} \ge \lambda_{21}$.
\begin{example}\label{nenita3}
Suppose that $\mu_1 = \mu_2 = \mu$. Let us define $r_1 = \lambda_{11} - \lambda_{21} \ge 0$, $r_2 = \lambda_{22} - \lambda_{12} \ge 0$ and the function $a_i(\mu)$ as $$a_i(\mu) = \sqrt{(\lambda_{i2}+\mu)^2+4\lambda_{i1}\mu}.$$
We know from Proposition \ref{WSS_fun_fix_ssi} that $r_2 \ge 2r_1$ is a necessary condition for $\tau^W_1 \ge_{lr} \tau^W_2$ to hold. Assume that $r_2 > 2r_1$. We will show that $\tau^W_1 \ge_{lr} \tau^W_2$ holds for $\mu$ sufficiently large.

Using (\ref{expr1a}) we have
\[a_2^2(\mu) - a_1^2(\mu) = r_2(\lambda_{12} + \lambda_{22}) + 2\mu(r_2 - 2 r_1) \ge 0.\]
It is easy to see that
\begin{eqnarray}
	\lim\limits_{\mu \rightarrow 0} (a_2(\mu) - a_1(\mu)) &=& r_2,\label{exa2}\\
	\lim\limits_{\mu \rightarrow \infty} (a_2(\mu) - a_1(\mu)) &=& r_2 - 2r_1 \le r_2,\label{exa3}\\
	\frac{d}{d\mu}(a_2(\mu) - a_1(\mu)) &=& \frac{\mu + \lambda_{22} + 2\lambda_{21}}{a_2(\mu)} - \frac{\mu + \lambda_{12} + 2\lambda_{11}}{a_1(\mu)}.\label{exa4}
\end{eqnarray}

Our interest is to find values of $\mu$ such that $r_2 - 2r_1 \ge a_2(\mu) - a_1(\mu)$ {because} then $\tau^W_1 \ge_{lr} \tau^W_2$. From (\ref{exa2}) and (\ref{exa3}), it is sufficient to prove that $a_2(\mu) - a_1(\mu)$ is increasing in a neighborhood of infinity for $r_2 - 2r_1 \ge a_2(\mu) - a_1(\mu)$ to hold in the same neighborhood. To prove this note that
\[ \lim\limits_{\mu \rightarrow \infty} \frac{a_1(\mu)}{a_2(\mu)}  = 1 > \frac{\lambda_{12} + 2\lambda_{11} + \mu}{\lambda_{22}+ 2\lambda_{21} + \mu},\]
where the last inequality is true due to $r_2 > 2r_1$.

Consider a sufficiently small  $\delta > 0$ and  a sufficiently large $M > 0$, such that if $\mu > M$ it holds that
\[ \frac{a_1(\mu)}{a_2(\mu)} \ge \frac{\lambda_{12} + 2\lambda_{11} + \mu}{\lambda_{22}+ 2\lambda_{21} + \mu} + \delta.\]
{Using (\ref{exa4}) it can be seen that this last inequality is equivalent to}
\[\frac{d}{d\mu}(a_2(\mu) - a_1(\mu)) \ge \delta \frac{\lambda_{22} + 2\lambda_{21} + \mu}{a_1(\mu)} > 0.\]
Then, for $\mu > M$ the function $a_2(\mu) - a_1(\mu)$ is increasing and consequently $\tau^W_1 \ge_{lr} \tau^W_2$.

Figure \ref{fig::figure1} shows a graph of $a_2(\mu) - a_1(\mu)$ with $\lambda_{11} = 1.5$, $\lambda_{22} = 3$, $\lambda_{12} = \lambda_{21} = 1$. In this case $a_2(\mu) - a_1(\mu) \le r_2 - 2r_1 = 1$ is satisfied {for $\mu \ge 11.25$ and consequently $\tau^W_1 \ge_{lr} \tau^W_2$ {for those values}. Note that the function $a_2(\mu) - a_1(\mu)$ is increasing for $\mu \ge 26.49$.}

\begin{figure}[t]
	\centering
	\includegraphics[width=0.6\linewidth]{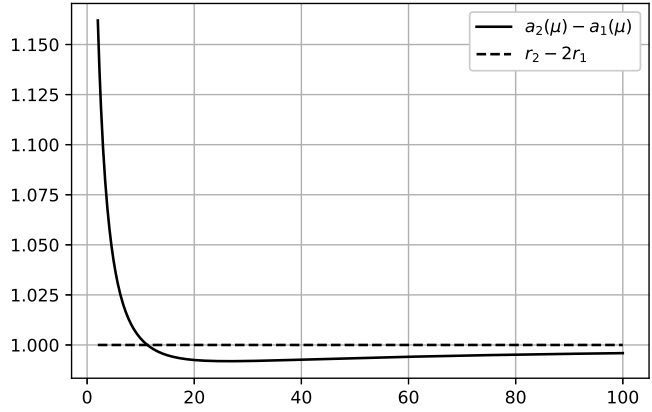}
	\caption{Graph of $a_2(\mu) - a_1(\mu)$ and the line $y = r_2 - 2r_1$ for $\lambda_{11} = 1.5$, $\lambda_{22} = 3$, $\lambda_{12} = \lambda_{21} = 1$.}
	\label{fig::figure1}
\end{figure}

\end{example}

\begin{remark}
	From Corollary \ref{WSS_fun_fix} and Example \ref{nenita3} we obtain that if $\mu_1 = \mu_2 = \mu$ is sufficiently large, then the condition $2(\lambda_{21} - \lambda_{11}) - \lambda_{12} + \lambda_{22} \ge 0$ is necessary and sufficient for $\tau^W_1 \ge_{lr} \tau^W_2$ to hold.
\end{remark}

Next, we assume that the units of both systems have stochastically equal lifetimes.
\begin{proposition}\label{WSS_rep_fix}
	Suppose that $\lambda_{1i} = \lambda_{2i} = \lambda_i$, for $i=1,2$. Then $\mu_1 \geq \mu_2$ if and only if $\tau^{W}_{1} \geq_{hr} \tau^{W}_{2}$.%, and we cannot improve hazard rate order to likelihood ratio order.
\end{proposition}

\begin{proof}

Let $r^{{W}}_i(t)$ be the hazard rate of $\tau^{W}_i$, we have \[r^{W}_{i}(t)=\frac{2\lambda_1(\lambda_1+\lambda_2)}{2\lambda_1+\lambda_2+\mu_i+a_{i}\coth(\frac{a_{i}t}{2})},\]
for $i=1,2$. {Note that $r^{W}_{{1}}(t) \leq r^{W}_{{2}}(t)$ is equivalent to
\[\mu_1 + a_{1}\coth\left(\frac{a_{1}t}{2}\right) \geq \mu_2 + a_{2}\coth\left(\frac{a_{2}t}{2}\right)\]
and this last inequality is fulfilled if and only if $\mu_1 \ge \mu_2$.}
\end{proof}

Under the hypothesis of Proposition \ref{WSS_rep_fix} the ordering $\tau^{W}_{1} \ge_{lr} \tau^{W}_{2}$ does not hold. To show that let us analyze the likelihood ratio
\[ \frac{\phi^{W}_{{1}}(t)}{\phi^{W}_{{2}}(t)} = \frac{a_1}{a_2} \frac{\exp\left\{\frac{\mu_2}{2} t\right\}\sinh(\frac{a_{1} t}{2})}{\exp\left\{\frac{\mu_1}{2} t\right\}\sinh(\frac{a_{2} t}{2})}.\]

{Note that}
\[ \left( \frac{\phi^{W}_{{1}}(t)}{\phi^{W}_{{2}}(t)} \right)' =_{sg} g(t) = \mu_2 +a_{1} \coth\left(\frac{a_{1}t}{2}\right)-\mu_1 - a_{2} \coth\left(\frac{a_{2}t}{2}\right).\]
Moreover, $g(t)$ is increasing and
\begin{eqnarray*}
	\lim_{t\rightarrow 0} g(t) &=& \mu_2 - \mu_1 < 0,\\
	\lim_{t\rightarrow \infty} g(t) &=& \mu_2 - \mu_1 +a_{1} -a_{2} > 0.
\end{eqnarray*}
So, the likelihood ratio $\displaystyle \frac{\phi^{W}_{{1}}(t)}{\phi^{W}_{{2}}(t)}$ is not increasing in $[0,\infty)$, and as a consequence $\tau^{W}_{1} \ngeq_{lr} \tau^{W}_{2}$.

%Figure \ref{fig::figure3} shows the graphs of the hazard rates and the likelihood ratio of two warm standby systems with $\lambda_{11} = \lambda_{21} = 2$, $\lambda_{12} = \lambda_{22} = 1$, $\mu_1 = 5$ and $\mu_2 = 3$. It can be seen that $r^W_1(t) \le r^W_2(t)$ but $\phi^W_1(t)/\phi^W_2(t)$ is not increasing. Thus $\tau^W_1 \ngeq_{lr} \tau^W_2$ although $\tau^W_1 \ge_{hr} \tau^W_2$.
%
%\begin{figure}[ht]
%	\centering
%	\includegraphics[width=0.6\linewidth]{script/2134.eps}
%	\caption{Graphs of likelihood ratio $\phi^W_1/\phi^W_2$ and the hazard rates $r^W_1$ and $r^W_2$ for $\lambda_{11} = \lambda_{21} = 2$, $\lambda_{12} = \lambda_{22} = 1$, $\mu_1 = 5$ and $\mu_2 = 3$.}
%	\label{fig::figure3}
%\end{figure}

From Corollary \ref{WSS_fun_fix} and Proposition \ref{WSS_rep_fix} we obtain

\begin{corollary}\label{WSS_general}
	If $\lambda_{11} \leq \lambda_{21}$, $\mu_1 \geq \mu_2$ and $2(\lambda_{21} - \lambda_{11}) \ge \lambda_{12} - \lambda_{22},$
	then $\tau^{W}_1 \geq_{hr} \tau^{W}_2$.
\end{corollary}

\section{{Markovian} cold standby system}\label{subsec::cold_markovian}

Let $\tau^{C}$ be the lifetime of a {cold standby system} with expected lifetime $1/\lambda$ for the principal unit and with expected repair time $1/\mu$. Let $\Phi^{C}(t)$, $\phi^{C}(t)$ and $r^{C}(t)$  be the survival function, the probability density function and the hazard rate functions of $\tau^{C}$, respectively.
Taking $\lambda_2 = 0$ in (\ref{distr_WSS}), (\ref{dens_WSS}) and (\ref{hazard_WSS}) we obtain
\begin{eqnarray*}
	\Phi^{{C}}(t) &=& \exp\left\{-\frac{(2\lambda+\mu)t}{2}\right\}\left[\cosh\left(\frac{bt}{2}\right) + \frac{2\lambda+\mu}{b}\sinh\left(\frac{bt}{2}\right)\right],\\
	\phi^{{C}}(t) &=& \exp\left\{-\frac{(2\lambda+\mu)t}{2}\right\} {\frac { 2\lambda^2 }{b}\sinh \left( \frac{bt}{2} \right)},\\
	r^{C}(t) &=& \frac{2\lambda^2}{b\coth(\frac{bt}{2})+(2\lambda+\mu)},
\end{eqnarray*}
where $b = \sqrt{\mu(4\lambda + \mu)}$.

Let $\tau^{C}_i$ be the lifetime of a  {cold standby system} with expected lifetime $1/\lambda_{i}$ for the principal unit and expected repair time $1/\mu_i$, for $i=1,2$, respectively. Let $b_i = \sqrt{\mu_i^2 + 4\lambda_i \mu_i}$, for $i=1,2$. From Proposition \ref{WSS_fun_fix_ssi}, taking $\lambda_{i2} = 0$, for $i=1,2$, we have the following result to compare the lifetimes of two  {cold standby systems} in the likelihood ratio order.

\begin{proposition}\label{nenita2}
 The relation  $\tau_1^{C} \geq_{lr} \tau_2^{C}$ is satisfied if and only if  $2(\lambda_2 - \lambda_1) + \mu_2 - \mu_1 \geq (b_2 - b_1)^+$.
 \end{proposition}

From Proposition \ref{nenita2} a necessary condition to obtain $\tau^{C}_1 \geq_{lr} \tau^{C}_2$ is
\( 2(\lambda_2 - \lambda_1)  \geq \mu_1 - \mu_2. \)
Under $\tau^{C}_1 \geq_{lr} \tau^{C}_2$, if we suppose $\lambda_1 \geq \lambda_2$, then necessarily $\mu_1 \leq \mu_2$ and thus as we will see in Proposition \ref{prop8}, $\tau^{C}_1 \leq_{hr} \tau^{C}_2$ holds, and consequently $\tau^{C}_1 \ngeq_{lr} \tau^{C}_2$. Hence, $\lambda_1 \leq \lambda_2$ is a necessary condition to obtain  $\tau^{C}_1 \geq_{lr} \tau^{C}_2$. {Note that this result and the one which we have obtained for the warm standby system are not similar {because Example \ref{nenita3} shows} that $\lambda_{11} \le \lambda_{21}$ is not a necessary condition for $\tau^W_1 \ge_{lr} \tau^W_2$ to hold.}

As a consequence of Proposition \ref{nenita2} we have

\begin{corollary}
	If $2(\lambda_2 - \lambda_1) \geq \mu_1 - \mu_2$ and $\mu_1 \geq \mu_2$, then $\tau_1^{C} \geq_{lr} \tau_2^{C}$.
\end{corollary}

\begin{proof}

The proof is similar to {that} of Corollary \ref{WSS_fun_fix} taking $\lambda_{12} = \lambda_{22} = 0$ and noting that (\ref{expr2}) is nonnegative when $\mu_1 \geq \mu_2$. The inequality $\lambda_1 \le \lambda_{2}$ ( $\lambda_{11} \le \lambda_{21}$ in the proof of Corollary \ref{WSS_fun_fix}) is not necessary due to $\lambda_{12} = \lambda_{22} = 0$.
\end{proof}

{The following example shows} that the inequality $\mu_1 \ge \mu_2$ is not necessary to obtain $\tau_1^{C} \geq_{lr} \tau_2^{C}$.

\begin{example}\label{nenita16}
	Let us suppose that $\mu_1 \le \mu_2$ and $\lambda_1 \le \lambda_2$. Then $b_1 \le b_2$ and, using Proposition \ref{nenita2}, we only {need to find out conditions for $2(\lambda_2 - \lambda_1) + \mu_2 - \mu_1 \geq b_2 - b_1$ to hold}. But $b_1 \ge \mu_1$ and
	\begin{eqnarray*}
		b_2 &=& \sqrt{\mu_2^2 + 4\lambda_2 \mu_2}\\
			&=& \sqrt{(\mu_2 + \lambda_2)^2 + 2 \lambda_2 \mu_2 - \lambda_2^2}\\
			&\le& \mu_2 + \lambda_2 \;\; \mbox{ if } \; \; 2\mu_2 \leq \lambda_2.
	\end{eqnarray*}
	So, $b_2 - b_1 \le \mu_2 + \lambda_2 - \mu_1$ and it is sufficient to check that \(2(\lambda_2 - \lambda_1) + \mu_2 - \mu_1 \geq \lambda_2 + \mu_2 - \mu_1,\) or equivalently, that $\lambda_2 \ge 2 \lambda_1$. Finally, we {get} $\tau_1^{C} \geq_{lr} \tau_2^{C}$ if
	\(
		\mu_1 \le \mu_2 \mbox{ and } \lambda_2 \ge 2\max\{\lambda_1, \mu_2\}.
	\)
	
	% Let us consider two cold standby systems with $\lambda_{1} = \mu_1 = 1$, $\lambda_2 = 5$ and $\mu_2 = 2$. For these values (\ref{cond2}) holds and $\tau_1^{C} \geq_{lr} \tau_2^{C}$, but $\mu_1 \le \mu_2$. Figure \ref{fig::figure2} shows that $\phi^C_1/\phi^C_2$ is increasing although $\mu_1 = 1 < 2 = \mu_2$.
	
	%  \begin{figure}[ht]
	%  	\centering
	%  	\includegraphics[width=0.6\linewidth]{script/example2}
	%  	\caption{Graph of the likelihood ratio $\phi^C_1/\phi^C_2$ for $\lambda_{1} = \mu_1 = 1$, $\lambda_2 = 5$ and $\mu_2 = 2$.}
	%  	\label{fig::figure2}
	%  \end{figure}
	
\end{example}

We will refer to the cold standby system with lifetime $\tau^{C}_i$ as the system $i${, for $i=1,2$.}
In particular, when the units of the system $1$ have stochastically greater lifetimes than the units of the system $2$ ($\lambda_1 \leq \lambda_2$) and they have stochastically equal repair times ($\mu_1 = \mu_2$), we obtain $\tau_1^{C} \geq_{lr} \tau_2^{C}$. When $\lambda_1 = \lambda_2$, we {get an ordering between the lifetimes of both systems} in the sense of the hazard rate, as a consequence of Proposition \ref{WSS_rep_fix}.

\begin{proposition}\label{CSS_fix_func}
	Suppose that $\lambda_{1} = \lambda_{2}$. Then $\mu_1 \geq \mu_2$ if and only if $\tau^{C}_{1} \geq_{hr} \tau^{C}_{2}$.
\end{proposition}

Notice that the {ordering $\tau^C_1 \ge_{lr} \tau^C_2$} does not hold under the assumptions of Proposition \ref{CSS_fix_func}, as a consequence of the analogous result for the warm standby system, obtained in Section \ref{subsecc::warm}.
The following result is related to the hazard rate order.
\begin{proposition}\label{prop8}
	If $\lambda_1 \leq \lambda_2$ and $\displaystyle \frac{\mu_1}{\mu_2} \geq \left(\frac{\lambda_1}{\lambda_2}\right)^2$, then $\tau_1^{C} \geq_{hr} \tau_2^{C}$
\end{proposition}
\begin{proof}

We must prove
\[\frac{2\lambda_1^2}{b_1\coth(\frac{b_1t}{2})+(2\lambda_1+\mu_1)} \geq \frac{2\lambda_2^2}{b_2\coth(\frac{b_2t}{2})+(2\lambda_2+\mu_2)},\]
which is equivalent to \( \lambda_1^2\left({b_2\coth\left(\frac{b_2t}{2}\right)+2\lambda_2+\mu_2}\right) \leq \lambda_2^2\left({b_1\coth\left(\frac{b_1t}{2}\right)+2\lambda_1+\mu_1}\right) \).
{From the assumptions}, it is easy to see that
\begin{equation}\label{eq_prop8_1}
	\lambda_1^2 (2\lambda_2+\mu_2) \leq \lambda_2^2 (2\lambda_1+\mu_1).
\end{equation}
Thus, it is sufficient to check the inequality
\(\lambda_1^2b_2 \coth\left(\frac{b_2t}{2}\right) \leq \lambda_2^2 b_1\coth\left(\frac{b_1t}{2}\right),\)
or equivalently
\[
	\frac{b_2 \coth\left(\frac{b_2t}{2}\right)}{b_1\coth\left(\frac{b_1t}{2}\right)} \leq \frac{\lambda_2^2}{\lambda_1^2}.
\]
By Lemma \ref{LEMON}, the function \( \displaystyle \frac{b_2 \coth\left(\frac{b_2t}{2}\right)}{b_1\coth\left(\frac{b_1t}{2}\right)}\) is monotone, and therefore we only need to verify that
\[
	\max\left\{ \lim\limits_{t \rightarrow 0}\frac{b_2 \coth\left(\frac{b_2t}{2}\right)}{b_1\coth\left(\frac{b_1t}{2}\right)}, \, \lim\limits_{t \rightarrow \infty}\frac{b_2 \coth\left(\frac{b_2t}{2}\right)}{b_1\coth\left(\frac{b_1t}{2}\right)} \right\} \leq \frac{\lambda_2^2}{\lambda_1^2}.
\]
Then, it is sufficient to prove the inequality $\lambda_1^2 b_2 \leq \lambda_2^2 b_1$. Now, taking squares in both sides of the last inequality, we obtain
\begin{equation}\label{eq_prop8_2}
	\lambda_1^4(\mu_2^2 + 4\lambda_2 \mu_2) \leq \lambda_2^4 (\mu_1^2 + 4\lambda_1 \mu_1).
\end{equation}
From $\lambda_1 \leq \lambda_2$ this last inequality holds, since $\displaystyle \frac{\mu_1}{\mu_2} \geq \left(\frac{\lambda_1}{\lambda_2}\right)^n$ for all $n \geq 2$.
\end{proof}

Note that if $\displaystyle \frac{\mu_1}{\mu_2} \geq \frac{\lambda_1}{\lambda_2}$, then {\(\displaystyle \frac{\mu_1}{\mu_2} \geq  \left(\frac{\lambda_1}{\lambda_2}\right)^2.\)} As a consequence $\tau_1^{C} \geq_{hr} \tau_2^{C}$ {holds} when the system $1$ has stochastically greater lifetimes of its units and stochastically smaller repair times of its units than the system $2$, i.e., $\lambda_1 \le \lambda_2$ and $\mu_1 \geq \mu_2$.

The condition $\displaystyle \frac{\mu_1}{\mu_2} \geq \left(\frac{\lambda_1}{\lambda_2}\right)^2$ is not necessary for the relation $\tau_1^{C} \ge_{hr} \tau_2^{C}$ to hold when $\lambda_1 \leq \lambda_2$. We show this in an example.

\begin{example}
	Let us consider $\displaystyle \frac{\mu_1}{\lambda_1^2} = 1 < 2 = \frac{\mu_2}{\lambda_2^2}$. From the proof of Proposition 6, for $\tau_1^{C} \ge_{hr} \tau_2^{C}$ to hold, it is sufficient to check inequalities (\ref{eq_prop8_1}) and (\ref{eq_prop8_2}) which are equivalent to $4/\lambda_1 - 8/\lambda_2 \geq 3$ and $1/\lambda_1 - 1/\lambda_2 \geq 1/2$, respectively. Since $\lambda_i$ is positive, for $i=1,2$, it is sufficient to check that $4/\lambda_1 - 8/\lambda_2 \geq 3$. Thus $\tau_1^{C} \ge_{hr} \tau_2^{C}$ when $\lambda_1 \leq \lambda_2$ and $4/\lambda_1 - 8/\lambda_2 \geq 3$.
\end{example}

Now we will find order relations between the lifetimes of {Markovian} warm and cold standby systems when at the initial instant there is a unit under repair and the other unit is working as principal. Let us denote the lifetimes of these systems as $\tau^{W*}$ and $\tau^{C*}$, respectively.
{
Let $\Phi^{W*}(t)$ $\left(\Phi^{C*}(t)\right)$, $\phi^{W*}(t)$ $\left(\phi^{C*}(t)\right)$ and $r^{W*}(t)$ $\left(r^{C*}(t)\right)$ be the survival, the probability density and the hazard rate functions of $\tau^{W*}$ $\left(\tau^{C*}\right)$, respectively. {The following expressions for these functions are derived using a similar reasoning to the one we used to get (\ref{distr_WSS}), (\ref{dens_WSS}) and (\ref{hazard_WSS}).}}
\begin{eqnarray*}
	\Phi^{W^*}(t) &=& \exp\left\{ -\frac{2 \lambda_1 + \lambda_2 + \mu}{2} t \right\} \left[ \cosh \left( \frac{a}{2}t \right) + \frac{\lambda_2 + \mu}{a}\sinh \left( \frac{a}{2}t \right) \right],\\
	\phi^{W^*}(t) &=& \exp\left\{ -\frac{2 \lambda_1 + \lambda_2 + \mu}{2} t \right\} \lambda_1 \left[ \cosh \left( \frac{a}{2}t \right) + \frac{\lambda_2 - \mu}{a}\sinh \left( \frac{a}{2}t \right) \right],\\
	r^{W^*}(t) &=& \lambda_1 \left[ \frac{a+(\lambda_2-\mu)\tanh\left(\frac{a}{2}t\right)}{a+(\lambda_2 + \mu)\tanh\left(\frac{a}{2}t\right)} \right].\\
\end{eqnarray*}
where $a=\sqrt{(\lambda_2+\mu)^2+4\lambda_1\mu}$. Now, taking $\lambda_2 = 0$ we have,
\begin{eqnarray*}
	\Phi^{C^*}(t) &=& \exp\left\{ -\frac{2 \lambda + \mu}{2} t \right\} \left[ \cosh \left( \frac{b}{2}t \right) + \frac{\mu}{b}\sinh \left( \frac{b}{2}t \right) \right],\\
	\phi^{C^*}(t) &=& \exp\left\{ -\frac{2 \lambda + \mu}{2} t \right\} \lambda \left[ \cosh \left( \frac{b}{2}t \right) - \frac{\mu}{b}\sinh \left( \frac{b}{2}t \right) \right],\\
	r^{C^*}(t) &=& \lambda \left[ \frac{b-\mu\tanh\left(\frac{b}{2}t\right)}{b+
			\mu\tanh\left(\frac{b}{2}t\right)} \right],
\end{eqnarray*}
where $b=\sqrt{\mu^2+4\lambda\mu}$.

Let $\tau^{W*}_i$ be the lifetime of a {warm standby system} when at the initial instant there is a unit under repair and the other is working as principal, with expected lifetime $1/\lambda_{i1}$ for the principal unit and  $1/\lambda_{i2}$ for the {standby} unit, and with expected repair time $1/\mu_i$, for $i=1,2$, respectively.
In a similar way to Propositions \ref{WSS_general} and \ref{prop8} we establish a hazard rate ordering between $\tau^{W*}_1$ and $\tau^{W*}_2$, and also between  $\tau^{C*}_1$ and $\tau^{C*}_2$.

\begin{proposition}\label{WSS^*_hr_prop}
	If $\lambda_{11} \leq \lambda_{21}$, $\mu_{1} \geq \mu_{2}$ and
	\begin{equation}\label{WSS^*_hyp}
	\frac{\lambda_{12}}{\lambda_{22}} \leq \frac{\mu_{1}}{\mu_{2}},
	\end{equation}
	 then $\tau^{W^*}_1 \geq_{hr}\tau^{W^*}_2$.
\end{proposition}

\begin{proof}

Let us consider the inequality,
\begin{eqnarray}
	\lambda_{11} \left[ \frac{a_1+(\lambda_{12}-\mu_1)\tanh\left(\frac{a_1}{2}t\right)}{a_1+(\lambda_{12}+
	\mu_1)\tanh\left(\frac{a_1}{2}t\right)} \right]
	 \leq \lambda_{21} \left[ \frac{a_2+(\lambda_{22}-\mu_2)\tanh\left(\frac{a_2}{2}t\right)}{a_2+(\lambda_{22}+
	\mu_2)\tanh\left(\frac{a_2}{2}t\right)} \right].\label{nenita1}
\end{eqnarray}

As $\lambda_{11} \le \lambda_{21}$, to prove (\ref{nenita1}) it is sufficient to prove the inequality
\begin{equation*}\label{WSS^*_hr}
\frac{a_1+(\lambda_{12}-\mu_1)\tanh\left(\frac{a_1}{2}t\right)}{a_1+(\lambda_{12}+
\mu_1)\tanh\left(\frac{a_1}{2}t\right)} \leq \frac{a_2+(\lambda_{22}-\mu_2)\tanh\left(\frac{a_2}{2}t\right)}{a_2+(\lambda_{22}+
\mu_2)\tanh\left(\frac{a_2}{2}t\right)}.
\end{equation*}
or equivalently,
\[
\left[a_1+(\lambda_{12}-\mu_1)\tanh\left(\frac{a_1}{2}t\right)\right] \left[a_2+(\lambda_{22}+
\mu_2)\tanh\left(\frac{a_2}{2}t\right)\right] \leq\] \[ \left[a_1+(\lambda_{12}+
\mu_1)\tanh\left(\frac{a_1}{2}t\right)\right] \left[a_2+(\lambda_{22}-\mu_2)\tanh\left(\frac{a_2}{2}t\right)\right].
\]

Now, to prove the last inequality it is sufficient to check the following ones
\begin{eqnarray}
	(\lambda_{12}-\mu_1)(\lambda_{22}+\mu_2) &\leq& (\lambda_{12}+\mu_1)(\lambda_{22}-\mu_2),\label{WSS^*_hr_1}\\
	a_1\mu_2\tanh\left(\frac{a_2}{2}t\right)  &\leq& a_2\mu_1 \tanh\left(\frac{a_1}{2}t\right).\label{WSS^*_hr_2}
\end{eqnarray}

After some transformations we obtain that (\ref{WSS^*_hr_1}) is equivalent to (\ref{WSS^*_hyp}).
Besides, (\ref{WSS^*_hr_2}) can be written as
\[h(t) = \frac{\tanh\left(\frac{a_2}{2}t\right)}{\tanh\left(\frac{a_1}{2}t\right)} \leq \frac{a_2 \mu_1}{a_1 \mu_2}.\]

Using Lemma \ref{LEMON} we can see that $h(t)$ is monotone, so (\ref{WSS^*_hr_2}) is equivalent to
\[\max\left\{ \lim\limits_{t \rightarrow 0}h(t),\lim\limits_{t \rightarrow \infty}h(t) \right\} = \max\left\{ \frac{a_2}{a_1}, 1 \right\} \leq \frac{a_2 \mu_1}{a_1 \mu_2}.\]

Due to $\mu_1 \geq \mu_2$ we only need to prove the inequality $a_1 \mu_2 \leq a_2 \mu_1$. For this it is sufficient to verify that
 \begin{eqnarray*}
 	[(\lambda_{12}+\mu_1)\mu_2 + (\lambda_{22}+\mu_2)\mu_1] [(\lambda_{12}+\mu_1)\mu_2 - (\lambda_{22}+\mu_2)\mu_1] &\le& 0,\\
 	4\mu_1\mu_2(\lambda_{11}\mu_2 - \lambda_{21}\mu_1) &\leq& 0.
 \end{eqnarray*}
But these inequalities follow from (\ref{WSS^*_hyp}) and $\lambda_{11}\mu_2 \leq \lambda_{21}\mu_1$.
\end{proof}

From (\ref{nenita1}), taking $t = 0$, it is easy to see that $\lambda_{11} \leq \lambda_{21}$ is a necessary condition for $\tau^{W^*}_1 \geq_{hr}\tau^{W^*}_2$ to hold.

As a particular case of Proposition \ref{WSS^*_hr_prop}, {the ordering $\tau^{W^*}_1 \geq_{hr}\tau^{W^*}_2$ holds} when $\lambda_{11} \leq \lambda_{21}$, $\lambda_{12} \leq \lambda_{22}$ and $\mu_{1} \geq \mu_{2}$.

Let us consider two {Markovian} {cold standby systems} when at the initial instant there is a unit under repair and the other unit is working as principal, with expected lifetime $1/\lambda_{i}$ for the principal unit and expected repair time $1/\mu_i$, for $i=1,2$, respectively.
Taking $\lambda_{12} = \lambda_{22} = 0$ in Proposition \ref{WSS^*_hr_prop} we obtain $\tau^{C^*}_1 \geq_{hr}\tau^{C^*}_2$ when $\lambda_{1} \leq \lambda_{2}$ and $\mu_{1} \geq \mu_{2}$.

%Furthermore, note that $\tau^{W} \geq_{lr} \tau^{W^*}$ and $\tau^{C} \geq_{lr} \tau^{C^*}$. Additionally,
%\begin{eqnarray*}
%	\lim\limits_{t \rightarrow \infty}r^W(t) &=& \lim\limits_{t \rightarrow \infty}r^{W*}(t) = \frac{2\lambda_1(\lambda_1+\lambda_2)}{2\lambda_1+\lambda_2+\mu+a},\\
%	\lim\limits_{t \rightarrow \infty}r^C(t) &=& \lim\limits_{t \rightarrow \infty}r^{C*}(t) = \frac{2\lambda^2}{2\lambda+\mu + b }.\\
%\end{eqnarray*}

\subsection{Aging classes}\label{subs::ageing_classes}

{
Let $X$ be a nonnegative random variable  and $t \ge 0$ a real number. The residual lifetime of $X$, denoted by $X_t$, is defined as
$X_t = (X-t| X>t)$.
A random variable $X$ with probability density function $f(x)$ is said to belong to the ageing class \textit{Increasing Likelihood Ratio} ($ILR$) if $f(x+t)/f(x)$ decreases in $x \geq 0$, for all $t \geq 0$. This condition is equivalent to $X_s \ge_{lr} X_t$ for $0\le s \le t$. It is well know that this ageing class is contained in other important ageing classes as \textit{Increasing Failure Rate} ($IFR$) and \textit{New Better than Used} ($NBU$).
The random variable $X$ is said to belong to the ageing class \textit{Decreasing Likelihood Ratio} ($DLR$) if $f(x+t)/f(x)$ increases in $x \geq 0$, for all $t \geq 0$. Also, if $X \in DLR$ then $X$ belongs to \textit{Decreasing Failure Rate} ($DFR$) and \textit{New Worst than Used} ($NWU$) ageing classes. For more details about the ageing classes see \cite{barlow_r._e._statistical_1981} and \cite{tadashi_dohi_orders_2002}.

Consider a warm standby system and a cold standby system. Suppose the lifetime of the principal unit and the repair time of the units of these systems are exponentially distributed with hazard rates $\lambda_1$ and $\mu$, respectively. Also, the lifetime of the standby unit of the warm standby system is an exponential random variable with hazard rate $\lambda_2$. The following result is related to the ageing classes the lifetimes of these systems are in.

\begin{proposition}\label{classes}
For all $\lambda_1$, $\lambda_2$ and $\mu$, $\tau^{W}, \tau^C \in ILR$ and $\tau^{W*}, \tau^{C*} \in DLR$.
\end{proposition}

\begin{proof}
First we will prove $\tau^{W} \in ILR$. Consider the ratio,
\[\frac{\phi^{W}(x+t)}{\phi^{W}(x)} = \exp\left\{-\frac{(2\lambda_1+\lambda_2+\mu)t}{2}\right\} \frac{\sinh \left( \frac{a(t+x)}{2} \right)}{\sinh \left( \frac{ax}{2} \right)}.\]

This last expression is decreasing in $x$ if and only if the function
\( \displaystyle g(x) = \frac{\sinh \left( \frac{a(t+x)}{2} \right)}{\sin \left( \frac{ax}{2} \right)},\)
decreases in $x$. Note that
\begin{eqnarray*}
	g(x) &=& \frac{ \sinh \left( \frac{at}{2} \right) \cosh \left( \frac{ax}{2} \right) +  \cosh \left( \frac{at}{2} \right) \sinh \left( \frac{ax}{2} \right)}{\sinh \left( \frac{ax}{2} \right)}\\
		&=&  \sinh \left( \frac{at}{2} \right) \coth \left( \frac{ax}{2} \right) + \cosh \left( \frac{at}{2} \right).
\end{eqnarray*}
Hence, $g(x)$ is a decreasing function in $x$.

Analogously to the previous proof, in order to prove $\tau^{W*} \in DLR$, it is sufficient to check that
\[g(x) = \frac{ \cosh \left( \frac{a}{2}(t+x) \right) + \frac{\lambda_2 - \mu}{a}\sinh \left( \frac{a}{2}(t+x) \right) }{\cosh \left( \frac{a}{2}x \right) + \frac{\lambda_2 - \mu}{a}\sinh \left( \frac{a}{2}x \right)},\]
is increasing in $x$.

After some transformations we get
\[g(x) = \cosh \left( \frac{a}{2}t \right) + \sinh \left( \frac{a}{2}t \right) \left[ \frac{a \tanh \left( \frac{a}{2}x \right) + \lambda_2 - \mu}{a  + (\lambda_2 - \mu) \tanh \left( \frac{a}{2}x \right)} \right].\]

Now, as $ \tanh \left( \frac{a}{2}x \right)$ is increasing in $x$, it is sufficient to prove that the function,
\[h(u) = \frac{a u + \lambda_2 - \mu}{a  + (\lambda_2 - \mu) u},\]
is also increasing in $u$. Deriving $h(u)$, we see that $h'(u) \geq 0$ is equivalent to
\[a^2 = (\lambda_2+\mu)^2+4\lambda_1\mu  \geq (\lambda_2 - \mu)^2,\]
and this inequality is equivalent to
\(4\mu(\lambda_1 + \lambda_2) \geq 0. \)

Taking $\lambda_2 = 0$ we obtain the analogous result for the {cold standby system} model.
\end{proof}

As a consequence of Proposition \ref{classes}, when $t$ increases the residual lifetimes $\tau^W_t = (\tau^W - t| \tau^W > t)$ and $\tau^C_t = (\tau^C - t| \tau^C > t)$ decrease in the sense of the likelihood ratio order and the residual lifetimes of  $\tau^{W*}$ and $\tau^{C*}$ increase in the same sense.

Finally, note that taking $\lambda_{12} =0$ in Corollary \ref{WSS_fun_fix} we obtain $\tau^C \ge_{lr} \tau^W$, for all $\lambda_1, \lambda_2$ and $\mu$.

%\begin{proposition}
%	For all $\lambda_1$, $\lambda_2$ and $\mu$, the relation $\tau^C \ge_{lr} \tau^W$ holds.
%\end{proposition}

\section{Cold standby systems. General distributions}\label{2unitGral}

Having discussed the Markovian cold standby system, the final section of this paper addresses the cold standby system from a broader perspective. We still assume that the system is composed of two units: $C_1$ and $C_2$. However, we now suppose that for $i=1,2$,\, $C_i$ has lifetime $X_i$ with a general distribution function $F_i(t)$ and density function $f_i(t)$, and repair time $Y_i$ with distribution function $G_i(t)$.
When $C_1$ fails, and $C_2$ is available, $C_1$ is immediately sent to reparation and $C_2$ takes its place. Likewise, when $C_2$ fails, and $C_1$ is available, $C_2$ is immediately sent to repair unit and $C_1$ takes its place. This process continues until one unit fails, while the other unit is being repaired.

{These systems, in general, are non-Markovian and their study using stochastic orders could be useful to decide which design is more useful or valuable in real-life two-units cold standby systems.}

Let us define the following random variables,
\begin{itemize}
	\item[$\tau_0^C$] lifetime of the system when at the initial instant $C_1$ starts to work and $C_2$ is waiting,
	\item[$\tau_1^C$] lifetime of the system when at the initial instant $C_1$ starts to be repaired and $C_2$ starts to work,
	\item[$\tau_2^C$] lifetime of the system when at the initial instant $C_2$ starts to work and $C_1$ starts its reparation.
	\item[$\tau_3^C$] lifetime of the system when at the initial instant $C_2$ starts to work and $C_1$ is waiting.
\end{itemize}

The following system of integral equations holds,
\begin{eqnarray*}
	\Phi_0^C(t) &=& P(X_1 > t) + \int_0^t \Phi_1^C(t - x)dF_1(x),\\
	\Phi_1^C(t) &=& P(X_2 > t) + \int_0^t G_1(x) \Phi_2^C(t - x)dF_2(x),\\
	\Phi_2^C(t) &=& P(X_1 > t) + \int_0^t G_2(x) \Phi_{{1}}^C(t - x)dF_1(x),\\
    \Phi_3^C(t) &=& P(X_2 > t) + \int_0^t \Phi_2^C(t - x)dF_2(x),
	%\Phi_3^C(t) &=& P(X_2 > t) + \int_0^t \overline{F}_1(x) \Phi_2^C(t - x)dG_2(x).
\end{eqnarray*}
where $\Phi^C_i(t)$ is the survival function of $\tau^C_i$, for $i=0,1,2,3$.

Applying the Laplace transformation to the previous system we get,
\begin{equation}
  \begin{array}{rcl}\label{sistem_int}
	\widehat{\Phi}_0^C(s) &=& \widehat{\overline{F}}_1(s) + \widehat{\Phi}_1^C(s) \left( 1 - s \widehat{\overline{F}}_1(s) \right), \\
	\widehat{\Phi}_1^C(s) &=& \widehat{\overline{F}}_2(s) +  \widehat{G_1f_2}(s) \widehat{\Phi}_2^C(s), \\
	\widehat{\Phi}_2^C(s) &=& \widehat{\overline{F}}_1(s) + \widehat{G_2f_1}(s) \widehat{\Phi}_1^C(s),  \\
    \widehat{\Phi}_3^C(s) &=& \widehat{\overline{F}}_2(s) + \widehat{\Phi}_2^C(s) \left( 1 - s \widehat{\overline{F}}_1(s)\right).
\end{array}
\end{equation}

{
Consider the following allocation problem: we want to decide which unit should start to work at the initial instant whereas the other is waiting in standby. Notice that this problem is equivalent to compare $\tau_0^C$ and $\tau_3^C$.

\begin{proposition}
  Suppose that $X_1 =_{st} X_2 =_{st} X$. If $Y_1 \le_{st} Y_2$, or $Y_1 \le_{icv}Y_2$ and $f(t)$, the density function of $X$, is decreasing, then $\tau_0^C \ge_{lt} \tau_3^C$.
\end{proposition}

\begin{proof}
	Let us denote by $F(t)$ the distribution function of $X$.
	From (\ref{sistem_int}) we get
	$$	\widehat{\Phi}_0^C(s) = \widehat{\overline{F}}(s) + \frac{\widehat{\overline{F}}(s)\left[1+\widehat{G_1f}(s)\right] }{1-\widehat{G_1f}(s)\widehat{G_2f}(s)}$$
	and
	$$	\widehat{\Phi}_3^C(s) = \widehat{\overline{F}}(s) + \frac{\widehat{\overline{F}}(s)\left[1+\widehat{G_2f}(s)\right] }{1-\widehat{G_2f}(s)\widehat{G_1f}(s)}. $$
	
    The ordering $\tau_0^C \ge_{lt} \tau_3^C$ is equivalent to $\widehat{G_1f}(s) \ge \widehat{G_2f}(s)$ and this inequality can be written as
    \begin{equation}\label{ineq}
      \int_0^\infty (\bar G_2(x)-\bar G_1(x)) e^{-sx} f(x) dx\ge 0,
    \end{equation}
    which is true when $Y_1 \le_{st} Y_2$. Now, if $Y_1 \le_{icv}Y_2$ and $f(t)$ is decreasing, the inequality (\ref{ineq}) is proved using part (b) of Lemma 7.1, p. 120 of \cite{barlow_r._e._statistical_1981}.
\end{proof}

Taking $s = 0$ in (\ref{sistem_int}), it is obtained,
\begin{equation}
  \begin{array}{rcl}\label{sistemSEL1}
	\mathbb{E}\left[\tau_0^C\right] &=& \mathbb{E}\left[X_1\right] + \mathbb{E}\left[\tau_1^C\right],  \\
	\mathbb{E}\left[\tau_1^C\right] &=& \mathbb{E}\left[X_2\right] + \mathbb{E}\left[\tau_2^C\right] P[X_2 > Y_1], \\
	\mathbb{E}\left[\tau_2^C\right] &=& \mathbb{E}\left[X_1\right] + \mathbb{E}\left[\tau_1^C\right] P[X_1 > Y_2],  \\
	\mathbb{E}\left[\tau_3^C\right] &=& \mathbb{E}\left[X_2\right] + \mathbb{E}\left[\tau_2^C\right].
\end{array}
\end{equation}

From (\ref{sistemSEL1}) we have
\[ \mathbb{E}[\tau_1^C] = \frac{\mathbb{E}\left[X_2\right]  + P[X_2 > Y_1] \mathbb{E}\left[X_1\right] }{1 - P[X_1 > Y_2] P[X_2 > Y_1]}.\]
Thus
\begin{equation}\label{analytic_expr}
   \mathbb{E}[\tau_0^C] = \mathbb{E}\left[X_1\right] + \frac{\mathbb{E}\left[X_2\right]  + P[X_2 > Y_1] \mathbb{E}\left[X_1\right] }{1 - P[X_1 > Y_2] P[X_2 > Y_1]}.
\end{equation}

Suppose now we want to analyze when $\mathbb{E}\left[\tau_0^C\right] \ge \mathbb{E}\left[\tau_3^C\right]$. Using (\ref{analytic_expr}), and the corresponding formula for $\mathbb{E}[\tau^C_3]$, this inequality can be written
\begin{eqnarray*}
	& & \mathbb{E}[X_1](1-P(X_2>Y_1)P(X_1>Y_2))+ \mathbb{E}[X_2] +  \mathbb{E}[X_1]P(X_2>Y_1)\\
	&\ge&  \mathbb{E}[X_2](1-P(X_1>Y_2)P(X_2>Y_1))+ \mathbb{E}[X_1] +  \mathbb{E}[X_2]P(X_1>Y_2),
\end{eqnarray*}
which is reduced to
\begin{equation}\label{veValdes}
\mathbb{E}[X_1]P(X_2>Y_1)(1-P(X_1>Y_2))\ge  \mathbb{E}[X_2]P(X_1>Y_2)(1-P(X_2>Y_1)).
\end{equation}

As a consequence of (\ref{veValdes}) we get
\begin{proposition}\label{allocation_gral}
	If $ \mathbb{E}[X_{1}] \ge  \mathbb{E}[X_{2}]$ and $P[X_{2} > Y_{1}] \ge P[X_{1} > Y_{2}]$, then $\mathbb{E}\left[\tau_0^C\right] \ge \mathbb{E}\left[\tau_3^C\right]$.
\end{proposition}

{The result of Proposition \ref{allocation_gral} is interesting because it seems contradictory. A critical moment for the system with lifetime $\tau_0^C$ is when $ C_2 $ is working due to the failure of $C_1$. If $ X_2 <Y_1 $, then system failure occurs. So it is logical to ask for $ P [X_ {2}> Y_ {1}] \ge P [X_ {1}> Y_ {2}] $ , even though the unit $ C_1 $ is principal in the system with lifetime $ \tau^C_0 $.}

\begin{example}
  The conditions in Proposition \ref{allocation_gral} hold, for instance, in the following cases:
\begin{enumerate}
  \item $E[X_1]=E[X_2]$, $X_1\le_{icv} X_2$, $Y_1=_{st} Y_2 =_{st} Y$ and $g(t)$, the density function of $Y$, is decreasing (this is true by part (b) of Lemma 7.1, p. 120 of \cite{barlow_r._e._statistical_1981}).
  \item $E[X_1]\ge E[X_2]$, $Y_i\equiv T_i$ and $F_1(T_2)\ge F_2(T_1)$. {Note that, if $X_1 \ge_{st}X_2$, then necessarily $T_1 \le T_2$.}
  \item {$Y_1 \equiv T_1$, $Y_2\equiv T_2$ and $X_1$, $X_2$ are exponentially distributed with hazard rates $\lambda_1$ and $\lambda_2$, respectively.}
  We proceed with the proof of this last case. Suppose that $X_i$, has mean $1/\lambda_i$, for $i=1,2$.
   Then, (\ref {veValdes}) is equivalent to
%$$ \frac{1}{\lambda_1} e^{-\lambda_2 T_1}(1- e^{-\lambda_1T_2}) \ge  \frac{1}{\lambda_2}  e^{-\lambda_1 T_2}(1- e^{-\lambda_2 T_1}),$$
%o equivalentemente
\begin{equation}\label{eqJoss}
   \frac{\lambda_2 e^{-\lambda_2 T_1}}{1 -  e^{-\lambda_2 T_1}} %\ge \frac{\lambda_2 e^{-\lambda_2 T_2}}{1 -  e^{-\lambda_2 T_2}}
  \ge \frac{\lambda_1 e^{-\lambda_1 T_2}}{1 -  e^{-\lambda_1 T_2}}.
\end{equation}
If $\lambda_1 \ge \lambda_2$ and $T_1 \le T_2$
the inequality (\ref{eqJoss}) holds because the function $q(x) = xe^{-x}/(1 - e^{-x})$ is decreasing for all $x \ge 0$ and $\lambda_2 T_1 \le \lambda_1 T_2$.

{
Also, if $\lambda_1 \le \lambda_2$, the inequality (\ref{eqJoss}) holds when $\lambda_2 T_1 \le \lambda_1 T_2$ since the function $u(x) = e^{-x}/(1 - e^{-x})$ is decreasing for all $x \ge 0$. Then,  $\mathbb{E}\left[\tau_0^C\right] \ge \mathbb{E}\left[\tau_3^C\right]$ if
%\begin{enumerate}
%  \item $\lambda_1 \ge \lambda_2$ and $T_1 \le T_2$, or
%  \item $\lambda_1 \le \lambda_2$ and $\displaystyle \frac{T_1}{T_2} \le \frac{\lambda_1}{\lambda_2}$.
%\end{enumerate}
%Those two conditions can be written in a compact expression as
$\displaystyle \frac{T_1}{T_2} \le \min\left\{1,  \frac{\lambda_1}{\lambda_2} \right\}$.}
\end{enumerate}
 \end{example}

{Assume $X_i$ and $Y_i$ exponentially distributed with hazard rates $\lambda_i$ and $\mu_i$ for $i=1,2$, respectively. Solving (\ref{sistem_int}) in this case we get,
\begin{equation}\label{LaplaceTau0}
	\widehat{\Phi}_0^C = \frac{(s + \lambda_2)(s + \lambda_1 + \mu_2)(s + \lambda_2 + \mu_1) + \lambda_1(s + \lambda_2)(s + \lambda_1 + \mu_2) + \lambda_1 \mu_1 (s + \lambda_1 + \lambda_2 + \mu_2)}{(s + \lambda_1)(s + \lambda_2)(s + \lambda_1 + \mu_2)(s + \lambda_2 + \mu_1) - \lambda_1 \lambda_2 \mu_1 \mu_2}.
\end{equation}}

Note that taking $s = 0$ in (\ref{LaplaceTau0}) we have
\begin{eqnarray}
	\mathbb{E}[\tau^C_0] &=& \frac{\mu_1 }{\lambda_1 \lambda_2 + \lambda_1 \mu_1 + \lambda_2 \mu_2} + \frac{\lambda_1 \lambda_2( \lambda_1 + \lambda_2+ \mu_1 + \mu_2) + \lambda_1\mu_1(\lambda_1 + \mu_2) + \lambda_2 \mu_2 (\lambda_2 + \mu_1)}{\lambda_1 \lambda_2(\lambda_1 \lambda_2 + \lambda_1 \mu_1 + \lambda_2 \mu_2)}.\label{eq2}
\end{eqnarray}
In (\ref{eq2}) the denominator of the first summand and the second summand are symmetric as a function of $\lambda_1$, $\mu_1$ and $\lambda_2$, $\mu_2$. Thus, $\mu_1 > \mu_2$ if and only if $\mathbb{E}\left[\tau^C_0\right] > \mathbb{E}\left[\tau^C_3\right]$, {i.e. the ordering between the mean lifetimes of $\tau^C_0$ and $\tau^C_3$ does not depend on the mean lifetimes of the units, but only depends on their mean repair times.}

In order to have $\widehat{\Phi}^C_0(s) \ge \widehat{\Phi}^C_3(s)$, because of the symmetry of the denominator of (\ref{LaplaceTau0}), it is sufficient to verify that the following inequality holds
\begin{eqnarray}
% \nonumber % Remove numbering (before each equation)
   & & (s + \lambda_2)(s + \lambda_1 + \mu_2)(s + \lambda_2 + \mu_1) + \lambda_1(s + \lambda_2)(s + \lambda_1 + \mu_2) + \lambda_1 \mu_1 (s + \lambda_1 + \lambda_2 + \mu_2) \nonumber\\
   & \ge & (s + \lambda_1)(s + \lambda_2 + \mu_1)(s + \lambda_1 + \mu_2) + \lambda_2(s + \lambda_1)(s + \lambda_2 + \mu_1) + \lambda_2 \mu_2 (s + \lambda_2 + \lambda_1 + \mu_1). \label{ordenLT}
\end{eqnarray}

Let us analyze what happen when $\lambda_1 = \lambda_2$.
\begin{proposition}\label{orden_lt1}
	If $\lambda_1 = \lambda_2 = \lambda$, then $\mu_1 \ge \mu_2$ if and only if $\tau_0^C \ge_{lt} \tau_3^C$.
\end{proposition}
\begin{proof}
Due to the symmetry in (\ref{ordenLT}) it is enough to check that
\[
    (s + \lambda)(s + \lambda + \mu_2) +  \mu_1 (s + 2\lambda + \mu_2)
    \ge
    (s + \lambda)(s + \lambda + \mu_1) + \mu_2 (s + 2\lambda + \mu_1).
\]
But this inequality is equivalent to
\(
    \mu_2 (s + \lambda) + \mu_1 (s + 2\lambda)
    \ge
    \mu_1 (s + \lambda) + \mu_2 (s + 2\lambda),
\)
which is satisfied if and only if $\mu_1 \ge \mu_2$.
\end{proof}

It is natural to ask what happens when $\mu_1 = \mu_2$. Surprisingly, we get the following result,
\begin{proposition}\label{orden_lt1}
	If $\mu_1 = \mu_2$ and $\lambda_1$, $\lambda_2$ are arbitrary, then $\tau_0^C =_{st} \tau_3^C$.
\end{proposition}
\begin{proof}
  Note we need to check that,
\[
    (s + \lambda_2)(s + \lambda_1 + \mu)(s + \lambda_2 + \mu) + \lambda_1(s + \lambda_2)(s + \lambda_1 + \mu) + \lambda_1 \mu (s + \lambda_1 + \lambda_2 + \mu)
\]
\[
    =
    (s + \lambda_1)(s + \lambda_2 + \mu)(s + \lambda_1 + \mu) + \lambda_2(s + \lambda_1)(s + \lambda_2 + \mu) + \lambda_2 \mu (s + \lambda_2 + \lambda_1 + \mu),
\]
or equivalently that
\[
    \lambda_2(s + \lambda_1 + \mu)(s + \lambda_2 + \mu) + \lambda_1(s + \lambda_2)(s + \lambda_1)  + \lambda_1\mu s + \lambda_1 \mu (s + \lambda_1 + \lambda_2 + \mu)
\]
\[
    =
     \lambda_1(s + \lambda_2 + \mu)(s + \lambda_1 + \mu) + \lambda_2(s + \lambda_1)(s + \lambda_2)+ \lambda_2\mu s + \lambda_2 \mu (s + \lambda_2 + \lambda_1 + \mu).
\]
This last inequality can be written as
\[
    (\lambda_2- \lambda_1)\left[ (s + \lambda_1 + \mu)(s + \lambda_2 + \mu) - (s + \lambda_1)(s + \lambda_2)  - \mu s - \mu (s + \lambda_1 + \lambda_2 + \mu)\right] = 0,
\]
which is trivially true.

\end{proof}

Consider two cold standby systems with lifetimes $\tau^C_{(1)}$ and $\tau^C_{(2)}$. Suppose that the unit which starts to work at the initial instant in {the} system with lifetime $\tau^C_{(i)}$ has lifetime $X_{i1}$ with distribution function $F_{i1}(t)$ and repair time $Y_{i1}$ with distribution function $G_{i1}(t)$, and the other unit (waiting at the initial instant) has lifetime $X_{i2}$ with distribution function $F_{i2}(t)$ and repair time $Y_{i2}$ with distribution function $G_{i2}(t)$, for $i = 1,2$. It is not difficult to check that (\ref{analytic_expr}) is increasing as a function of  $ \mathbb{E}[X_1]$, $ \mathbb{E}[X_2]$, $P[X_1 > Y_2]$ and $P[X_2 > Y_1]$. Thus,
\begin{proposition}\label{expectedValue_gral}
	If $ \mathbb{E}[X_{11}] \ge  \mathbb{E}[X_{21}]$, $ \mathbb{E}[X_{12}] \ge  \mathbb{E}[X_{22}]$, $P[X_{11} > Y_{12}] \ge P[X_{21} > Y_{22}]$ and $P[X_{12} > Y_{11}] \ge P[X_{22} > Y_{21}]$, then $\mathbb{E}\left[\tau_{(1)}^C\right] \ge \mathbb{E}\left[\tau_{(2)}^C\right]$.
\end{proposition}

As a consequence of Proposition \ref{expectedValue_gral}, when $X_{ij}$ and $Y_{ij}$ are exponentially distributed with hazard rates $\lambda_{ij}$ and $\mu_{ij}$, respectively, for $i,j \in \{1,2\}$, we have
\begin{proposition}
	If $\lambda_{11} \le  \lambda_{21}$, $\lambda_{12} \le  \lambda_{22}$, $\displaystyle \frac{\lambda_{11}}{\mu_{12}} \le \frac{\lambda_{21}}{\mu_{22}}$ y $\displaystyle \frac{\lambda_{12}}{\mu_{11}} \le \frac{\lambda_{22}}{\mu_{21}}$, then $\mathbb{E}\left[\tau_{(1)}^C\right] \ge \mathbb{E}\left[\tau_{(2)}^C\right]$.
\end{proposition}
Of course, if $\lambda_{11} \le  \lambda_{21}$, $\lambda_{12} \le  \lambda_{22}$, $\mu_{11} \ge \mu_{21}$ and $\mu_{12} \ge \mu_{22}$, then $\mathbb{E}\left[\tau_{(1)}^C\right] \ge \mathbb{E}\left[\tau_{(2)}^C\right]$.

\section*{Acknowledgment}

The authors are very grateful to the referees for their valuable comments which helped to
improve the presentation of this article.

\bibliographystyle{apalike}

\begin{thebibliography}{}
	
	\bibitem[Bao and Cui, 2012]{bao2012study}
	Bao, X. and Cui, L. (2012).
	\newblock A study on reliability for a two-item cold standby {M}arkov
	repairable system with neglected failures.
	\newblock {\em Communications in Statistics-Theory and Methods},
	41(21):3988--3999.
	
	\bibitem[Barlow and Proschan, 1981]{barlow_r._e._statistical_1981}
	Barlow, R.~E. and Proschan, F. (1981).
	\newblock {\em Statistical {T}heory of {R}eliability and {L}ife {T}esting:
		{P}robability {M}odels.}
	\newblock To Begin With, Silver Spring, MD.
	
	\bibitem[Belzunce et~al., 2015]{belzunce2015introduction}
	Belzunce, F., Riquelme, C.~M., and Mulero, J. (2015).
	\newblock {\em An introduction to stochastic orders}.
	\newblock Academic Press.
	
	\bibitem[Brito et~al., 2011]{gerandy_brito_hazard_2011}
	Brito, G., Zequeira, R.~I., and Valdés, J.~E. (2011).
	\newblock On the hazard rate and reversed hazard rate orderings in
	two-component series systems with active redundancies.
	\newblock {\em Statistics \& Probability Letters}, 81(2):201 -- 206.
	
	\bibitem[Chen et~al., 2017]{chenCSTM}
	Chen, J., Zhang, Y., Zhao, P. and Zhou, S. (2017).
	\newblock Allocation strategies of standby redundancies in series/parallel
	system.
	\newblock {\em Communications in Statistics - Theory and Methods}, 0(0):1--17.
	
	\bibitem[Chowdhury and Kundu, 2017]{kundu2017}
	Chowdhury, S. and Kundu, A. (2017).
	\newblock Stochastic comparison of parallel systems with log-{L}indley
	distributed components.
	\newblock {\em Operations Research Letters}, 45:199--205.
	
	\bibitem[Gnedenko and Ushakov, 1995]{boris_v._gnedenko_probabilistic_1995}
	Gnedenko, B. and Ushakov, I.~A. (1995).
	\newblock {\em Probabilistic Reliability Engineering}.
	\newblock Wiley-Interscience.
	
	\bibitem[Hazra and Nanda, 2017]{nandaCSTM}
	Hazra, N.~K. and Nanda, A.~K. (2017).
	\newblock General standby allocation in series and parallel systems.
	\newblock {\em Communications in Statistics - Theory and Methods},
	46(19):9842--9858.
	
	\bibitem[Li and Cao, 1993]{li93}
	Li, W. and Cao, J. (1993).
	\newblock The limiting distribution of the residual lifetime of a {M}arkov
	repairable system.
	\newblock {\em Reliability Engineering \& System Safety}, 41(2):103--105.
	
	\bibitem[Nakagawa, 2002]{tadashi_dohi_stochastic_2002}
	Nakagawa, T. (2002).
	\newblock Two-unit redundant models.
	\newblock In Osaki, S., editor, {\em Stochastic Models in Reliability and
		Maintenance}, chapter~7. Springer-Verlag Berlin Heidelberg, 1 edition.
	
	\bibitem[Ohnishi, 2002]{tadashi_dohi_orders_2002}
	Ohnishi, M. (2002).
	\newblock Stochastic orders in reliability theory.
	\newblock In Osaki, S., editor, {\em Stochastic Models in Reliability and
		Maintenance}, chapter~2. Springer-Verlag Berlin Heidelberg, 1 edition.
	
	\bibitem[Shaked and Shanthikumar, 2007]{shaked_stochastic_2007}
	Shaked, M. and Shanthikumar, J.~G. (2007).
	\newblock {\em Stochastic orders}.
	\newblock Springer, New York.
	
	\bibitem[Valdés and Zequeira, 2003]{jose_e._valdes_optimal_2003}
	Valdés, J.~E. and Zequeira, R.~I. (2003).
	\newblock On the optimal allocation of an active redundancy in a two-component
	series system.
	\newblock {\em Statistics \& Probability Letters}, 63(3):325 -- 332.
	
	\bibitem[Valdés and Zequeira, 2006]{jose_e._valdes_optimal_2006}
	Valdés, J.~E. and Zequeira, R.~I. (2006).
	\newblock On the optimal allocation of two active redundancies in a
	two-component series system.
	\newblock {\em Operations Research Letters}, 34(1):49 -- 52.
	
	\bibitem[Wang, 2017]{wang2017}
	Wang, J. (2017).
	\newblock Stochastic comparison in mrl ordering for parallel systems with two
	exponential components.
	\newblock {\em Operations Research Letters}, 45(3):187 -- 190.
	
	\bibitem[Zhao et~al., 2013]{peng_zhao_allocation_2013}
	Zhao, P., Chan, P.~S., Li, L. and Ng, H. K.~T. (2013).
	\newblock On allocation of redundancies in two-component series systems.
	\newblock {\em Operations Research Letters}, 41(6):690 -- 693.
	
	\bibitem[Zhao et~al., 2012]{peng_zhao_optimal_2012}
	Zhao, P., Chan, P.~S. and Ng, H. K.~T. (2012).
	\newblock Optimal allocation of redundancies in series systems.
	\newblock {\em European Journal of Operational Research}, 220(3):673 -- 683.
	
\end{thebibliography}

\end{document}